\documentclass{conm-p-l}

\usepackage{enumerate}

\copyrightinfo{2010}{American Mathematical Society}

\newcommand{\triv}{\mathbf{1}}
\newcommand{\St}{\mathsf{St}}

\newcommand{\ch}{\vartheta}
\newcommand{\extres}{\epsilon}
\newcommand{\Res}{\mathrm{Res}}
\newcommand{\Ind}{\mathrm{Ind}}
\newcommand{\Sh}{\mathcal{S}}

\newcommand{\sgn}{\mathsf{sgn}}
\newcommand{\Tr}{\mathrm{Tr}}
\newcommand{\R}{\mathcal{R}}
\newcommand{\RR}{\mathbb{R}}
\newcommand{\Z}{\mathbb{Z}}
\newcommand{\SL}{\mathrm{SL}}
\newcommand{\PGL}{\mathrm{PGL}}
\newcommand{\GL}{\mathrm{GL}}
\newcommand{\ZZ}{\mathcal{Z}}

\newcommand{\tor}{\mathfrak{t}}
\newcommand{\lrc}[1]{\lceil #1 \rceil}
\newcommand{\g}{\mathfrak{g}}

\newcommand{\Stor}{S}
\newcommand{\ratk}{k}
\newcommand{\resk}{\kappa}

\newcommand{\ep}{\varepsilon}
\newcommand{\PP}{\mathcal{P}}
\newcommand{\p}{\varpi}
\newcommand{\val}{\mathop{val}}

\newcommand{\A}{\mathcal{A}}
\newcommand{\B}{\mathcal{B}}
\newcommand{\Aphi}{\Gamma}
\newcommand{\aphi}{\mathsf{a}}

\newcommand{\cind}{\textrm{c-}\mathrm{Ind}}

\newcommand{\mat}[1]{\left[ \begin{matrix} #1 \end{matrix} \right]}
\newcommand{\smat}[1]{\left[ \begin{smallmatrix} #1 \end{smallmatrix} \right]}

\theoremstyle{plain}
\newtheorem{theorem}{Theorem}[section]
\newtheorem{observation}[theorem]{Observation}
\newtheorem{scholium}[theorem]{Scholium}

\newtheorem{corollary}[theorem]{Corollary}

\theoremstyle{definition}

\theoremstyle{remark}

\numberwithin{equation}{section}

\begin{document}
\title[Patterns in Branching Rules for $\SL_2(\ratk)$]{Patterns in Branching Rules for Irreducible Representations of $\SL_2(\ratk)$, for $\ratk$ a $p$-adic field}
\author{Monica Nevins}
\address{Department of Mathematics and Statistics, University of Ottawa, Ottawa, Canada K1N 6N5}
\email{mnevins@uottawa.ca}
\thanks{This research is supported by a Discovery Grant from NSERC Canada.}
\dedicatory{To Paul Sally, Jr., with thanks for much inspiration.}
\date{\today}
\begin{abstract}  Building on prior work, we analyze the decomposition of the restriction of an irreducible representation of  $\SL_2(\ratk)$, for $\ratk$ a $p$-adic field of odd residual characteristic, to a maximal compact subgroup $K$.  The pattern of the decomposition varies between principal series and different supercuspidal representations, whereas the $K$-representations which occur in the ``tail end'' of these decompositions are precisely those occuring in the decomposition of depth-zero supercuspidal representations.  Various applications are considered.
\end{abstract}

\maketitle

\section{Introduction}

The representation theory of $\SL_2(\ratk)$ is one of the most completely and classically studied areas of $p$-adic representation theory, being the natural starting point for any new line of investigation.  It is so accessible that it can be presented within 46 well-written pages of the definition of $p$-adic numbers, as in \cite{SallyPhysics}.  A marvelous attribute of this theory, and part of its attraction, is that there is still so much to be said.   

The idea of branching rules is to explore the decomposition of the restriction of an irreducible representation $\pi$ of $G$ to an interesting subgroup $K$, with the intention of discovering new perspectives on $\pi$, or even conversely, towards better understanding the representation theory of $K$.   

When $G$ is a semisimple real Lie group and $K$ its unique (up to conjugacy) maximal compact subgroup, for example, this idea led to the fundamental  discovery of the unicity of lowest $K$-types \cite{Vogan} as a classifying tool.  Furthermore, recent work with the ATLAS project has inspired questions of calculability of these types, and reflections on the representation theory of $K$ \cite{Vogan2}.

In contrast to the real case, when $G$ is a semisimple $p$-adic group there are in general several conjugacy classes of maximal compact subgroups.  If $G$ is simply connected, for example, they  correspond to conjugacy classes of stabilizers of vertices of the associated Bruhat-Tits building \cite[\S3.2]{Tits}.  Furthermore, the representation theory of these compact groups---which often reduces to that of Lie groups over finite local rings, an area of growing interest in its own right---is surprisingly incomplete in all but a few cases.

One highly successful direction of inquiry has been to replace $K$ with a variety of smaller compact open subgroups (whose representation theory is more amenable to study, as per the work of Howe \cite{Howe}, for example).  This led to the theory of types \cite{BushnellKutzko, MoyPrasadKtypes}, in which the classification of irreducible admissible representations is reduced to that of certain pairs $(\rho, J)$ where $\rho$ is an irreducible representation of the (usually non-maximal) compact open subgroup $J$. 

On the other hand, the original question of understanding the branching rules for the restriction of a representation of $G$ to a maximal compact subgroup $K$ remains, and is of interest on a number of fronts.  It was studied, under the hypothesis (here and throughout) that the residual characteristic of $\ratk$ is odd, for $\GL_2(\ratk)$ in \cite{Hansen, Casselman} and for $\PGL_2(\ratk)$ in \cite{Silberger,Silberger2}, where the goal was often to obtain a deeper understanding of the harmonic analysis of these groups.   In his doctoral thesis under Paul Sally, Jr. \cite{Tuvell}, Walter Tuvell determined the branching rules for principal series representations of $\SL_2(\ratk)$ as a step towards explicitly calculating the Eisenstein integrals, which can be used to compute Fourier transforms of non-invariant distributions.  Although successful, this proved to be an arduous task.


The branching rules for principal series of $\GL_3(\ratk)$ were partially determined in \cite{NC1,NC2}; obtaining a complete answer would go a long way towards the classification of the irreducible representations of $K=\GL_3(\R)$.

The branching rules for principal series of $\SL_2(\ratk)$, restricted to any maximal compact subgroup, were also determined in \cite{Nevins}, taking advantage of the classification of irreducible representations of $K = \SL_2(\R)$ given by Shalika in his thesis \cite{Shalika}.   Similarly, the branching rules for supercuspidal representations of $\SL_2(\ratk)$ were determined in \cite{Nevins2}.  Since all representations of $\SL_2(\ratk)$ fall into one of these categories, this allows one to complete the discussion for the group $\SL_2(\ratk)$ (when $\ratk$ has odd residual characteristic) in the present work.

Our main result uses the description of the branching rules from \cite{Nevins, Nevins2} to demonstrate the commonalities and differences of the components occuring in the branching rules for irreducible admissible representations of $\SL_2(\ratk)$, as made precise in Observations~\ref{O:depthzero} and \ref{O:posdepth}, Theorem~\ref{T:tailends} and its corollaries.  We discuss further applications in greater detail in Sections~\ref{SS:42} through \ref{SS:last}.  

The study of $\SL_2(\ratk)$ (and of principal series of $\GL_3(\ratk)$) has identified the next steps in this investigation of branching rules.  That is, although principal series have the advantage of a simple classification, the more direct realization of the $K$-irreducibles occuring in the branching rules for $\SL_2(\ratk)$ arises from the consideration of the supercuspidal representations.  Yu \cite{Yu} has given a construction of supercuspidal representations which applies to all connected reductive $p$-adic groups which split over a tamely ramified extension.   Kim \cite{Kim} has shown this construction to be exhaustive in many cases.  This construction is the one followed here and in \cite{Nevins2} to describe the supercuspidal representations of $\SL_2(\ratk)$, as it offers a potential template for extension to the general case.

In Section~\ref{S:reps} we set our notation and summarize the representation theories of $\SL_2(\ratk)$ and $\SL_2(\R)$.  In Section~\ref{S:BR}, we present the relevant results from \cite{Nevins, Nevins2}, using a unified notation to allow for vivid comparison, while avoiding the technical details that were needed to derive these branching rules.     We conclude in Section~\ref{S:conclusion} with our main theorems and observations.

\subsection*{Acknowledgments}  My sincere thanks to Fiona Murnaghan, Loren Spice and Jeff Adler for sharing their expertise on supercuspidal representations and Bruhat-Tits theory with me.   Their encouragement and advice throughout the course of this project have been invaluable.

\section{Representations of $\SL_2(\ratk)$ and $\SL_2(\R)$} \label{S:reps}

Let $\ratk$ be a $p$-adic field  with residue field $\resk$ of characteristic $p$.  Denote its integer ring by $\R$ and its maximal ideal by $\PP$.  Denote the congruence subgroups of $\R^\times$ by $1 + \PP^n$ for $n > 0$.  Choose a nonsquare $\ep \in \R^\times$; if $-1$ is not a square then we choose $\ep = -1$.  Let $\p$ be a uniformizer of $\ratk$ and normalize the valuation $\val$ on $\ratk$ so that $\val(\p)=1$.    Let $\Psi$ denote an additive character of $\ratk$ such that $\Psi$ is nontrivial on $\R$ and trivial on $\PP$.  
For $\tau \in \{\ep,\p, \ep\p\}$, the character $\sgn_\tau$ of $\ratk^\times$ is defined by $\sgn_\tau(x)=(x,\tau)$ for all $x\in \ratk^\times$, where $(\cdot,\cdot)$ denotes the 2-Hilbert symbol. 
For $r \in\RR$, we denote by $\lrc{r}$ the least integer $n$ such that $n \geq r$.   The extended real numbers, used as indices in Moy-Prasad filtrations, are the set $\tilde{\RR} = \RR \cup (\RR+) \cup \infty$; we define $\lrc{r+}$ to be the least integer $n$ such that $n>r$.
We abuse notation by using the same letter and font to denote an algebraic group as its group of $\ratk$-points; in context it should cause no confusion.

\subsection{Irreducible representations of $\SL_2(\ratk)$}
Let $G = \SL_2(\ratk)$ and $K = \SL_2(\R)$.  Let $Z = \ZZ(G) = \{\pm I\}$.  Set $\eta = \smat{1 & 0 \\ 0 & \p}$; then representatives of the two distinct classes of maximal compact subgroups of $G$ are $K$ and $K^\eta = \eta K \eta^{-1}$.  One sees this from the building $\B = \B(G,\ratk)$ of $G$ over $\ratk$. 
Let $\Stor$ denote the diagonal split torus of $G$ and $\A = \A(G,\Stor,\ratk)$ the corresponding apartment in $\B$.  The apartment is one-dimensional and we make the usual choice of coordinates on $\A$ so that the stabilizer of $y=0$ is $K=G_0$ and that of $y=1$ is $K^\eta = G_{1}$.  These points are representatives of the two classes of vertices in $\B$.

For any $x\in \B(G,\ratk)$, Moy and Prasad \cite{MoyPrasadKtypes} defined filtration subgroups of the stabilizer subgroup $G_x$, indexed by the extended number system $\tilde{\RR}$.  For $\SL_2(\ratk)$ and $x$ in the standard apartment $\A$ one can describe these filtration subgroups quite simply: for $r>0$, $G_{x,r}$ consists of those matrices in $G$ of the form $\smat{1+\PP^{\lrc{r}} & \PP^{\lrc{r-x}}\\ \PP^{\lrc{r+x}} & 1+\PP^{\lrc{r}}}$. 

The \emph{depth} of an irreducible smooth representation $(\pi,V)$ of $G$ is defined as the least $r \geq 0$ such that there exists $x\in \B$ for which $V$ contains vectors invariant under $G_{x,r+}$.
Similarly, if $x\in \B$ is fixed and $(\pi,V)$ is a representation of $G_x$ then we may define its depth relative to $x$.  

By a classical theorem of Jacquet \cite{Cartier}, all representations of $\SL_2(\ratk)$ are either supercuspidal representations or else occur as subrepresentations of representations parabolically induced from a supercuspidal representation of a proper Levi subgroup; this latter is precisely the principal series representations.  A principal series representation of $\SL_2(\ratk)$ is given by the choice of a character $\chi$ of $\Stor \cong \ratk^\times$, which is extended trivially over the upper triangular Borel subgroup $P$, and then induced, via normalized parabolic induction, to $G$.  
If $\chi$ is of depth $r$, then so is $\Ind_P^G \chi$.

The supercuspidal representations of $\SL_\ell(\ratk)$ in the tame case were established by Kutzko and Sally in \cite{KutzkoSally} and by Moy and  Sally in \cite{MoySally}.   There exist many subtly different descriptions of the construction of supercuspidal representations, including a general one by Stevens \cite{Stevens} extending \cite{BushnellKutzko} to all classical groups, without restriction on residual characteristic.  The one we choose to follow here is due to J.K.~Yu \cite{Yu}.  This work is given additional expository treatment in \cite{KimMonograph} and \cite{HakimMurnaghan}.

One begins by describing the cuspidal representations of $\SL_2(\resk)$.  This is well-known; see for example, the book by Digne and Michel \cite{DigneMichel}.  Let $\extres$ denote the unique quadratic extension field of $\resk$, and $N \colon \extres \to \resk$ the norm map.  Then $\ker(N)$ has order $q+1$.  Each character $\omega$ of this group gives rise to a representation $\sigma = \sigma(\omega)$ of $\SL_2(\resk)$ via Deligne-Lusztig induction.  If $\omega^2 \neq 1$, then this representation is irreducible, cuspidal and of degree $q-1$.  If $\omega = \omega_0$, the unique nontrivial character of order two, then $\sigma(\omega_0)$ decomposes into the two remaining, inequivalent irreducible cuspidal representations as $\sigma_0^+ \oplus \sigma_0^-$.  

The depth-zero supercuspidal representations are precisely the representations of the form 
$$
\cind_K^G \sigma \qquad \textrm{or} \qquad \cind_{K^\eta}^G \sigma^\eta 
$$
where $\cind$ denotes compact induction, $\sigma$ denotes the inflation of an irreducible cuspidal representation of $\SL_2(\resk)$ to $K$, and $\sigma^\eta$ is the representation of $K^\eta$ given by $\sigma^\eta(x)=\sigma(\eta^{-1}x\eta)$ for all $x\in K^\eta$.  That these exhaust the set of depth-zero supercuspidal representations is a special case of a general fact established in \cite{MoyPrasadJacquet, Morrislevel0}.

For the positive depth case, we begin by choosing an anisotropic torus $T$ of $G$.  Then the apartment of $T$ (over a splitting field) intersects $\B$ in a point $y = y_T$.  One can choose representatives for the distinct equivalence classes of anisotropic tori from among tori of the form
$$
T = T_{\gamma_1,\gamma_2} = \left\{ \left. t(a,b):=\mat{a & b\gamma_1 \\ b\gamma_2 & a} \right| a,b\in \R, \det(t(a,b))=1\right\}
$$
for some pair $(\gamma_1,\gamma_2)$ such that $\sqrt{\gamma_1\gamma_2} \notin \ratk$.  One distinguishes a torus as ramified or unramified, depending on its splitting field $\ratk(\sqrt{\gamma_1\gamma_2})$.  We give a list of representatives, chosen so that the associated points $y$ lie in the fixed apartment $\A$, in Table~\ref{Table:tori}.  

\begin{table}
\begin{tabular}{|l||l|c|}
\hline && \\ [-1.5ex] 
Unramified tori & $\displaystyle T_{1,\ep}$ & $y=0$\\[1ex]
& $\displaystyle T_{\p^{-1},\ep \p} =  T_{1,\ep}^\eta$ & $y=1$\\[1ex]
\hline && \\ [-1.5ex]
Ramified tori & $\displaystyle T_{(1,\p)}$ & $y=\frac12$\\[1ex]
with splitting field $\ratk(\sqrt{\p})$ &  $\displaystyle T_{(\ep,\ep^{-1}\p)}$ (if $-1 \in (\ratk^\times)^2$) & $y = \frac12$\\[1ex]
\hline && \\ [-1.5ex]
Ramified tori & $\displaystyle T_{(1,\ep\p)}$ & $y=\frac12$\\[1ex]
with splitting field $\ratk(\sqrt{\ep\p})$ &  $\displaystyle T_{(\ep,\p)}$ (if $-1 \in (\ratk^\times)^2$) & $y = \frac12$\\[1ex] \hline 
\end{tabular}
\caption{Representatives of the equivalence classes of anisotropic tori in $\SL_2(\ratk)$.  There are four if $-1 \notin (\ratk^\times)^2$ (when $\ep=-1$), and six otherwise.} \label{Table:tori}
\end{table}

The Lie algebra of $T_{\gamma_1,\gamma_2}$, denoted $\tor_{\gamma_1,\gamma_2}$, is the one-dimensional subalgebra of $\g$ spanned by 
\begin{equation} \label{E:X}
X_T = X_{\gamma_1,\gamma_2} = \mat{0&\gamma_1\\\gamma_2 & 0}.
\end{equation}
For any $u\in \tilde{\RR}$, $u > 0$, and given $y$ corresponding to $T$ as in the table, the Moy-Prasad filtration subgroups of $T$ can be described by $T_u = \{ t(a,b) \mid a\in 1+\PP^{\lrc{u}}, b\gamma_1\in \PP^{\lrc{u - y}}\}$.  For any $u \in \tilde{\RR}$, the filtration on $\tor$ is given by $\tor_u = \{ cX_{\gamma_1,\gamma_2} \mid c\gamma_1 \in \PP^{\lrc{u - y}}\}$.


A positive depth supercuspidal representation is parametrized by a generic tamely ramified cuspidal $G$-datum, which in the $\SL_2$ case consists of an anisotropic torus $T$, together with the corresponding point $y \in \B$, and a \emph{generic} quasi-character $\phi$ of $T$ of some depth $r>0$.  If $T$ is unramified then $r\in \Z$ and otherwise $r\in \frac12 + \Z$.  One extends $\phi$ to a representation $\rho$ of $TG_{y,r/2}$ as described below.  The corresponding induced representation $\cind_{TG_{y,r/2}}^G \rho$ is of depth $r$, and the collection of all these exhaust the positive depth supercuspidal representations of $G$.  In \cite{HakimMurnaghan}, it is shown that two such supercuspidal representations are equivalent if and only if their data are $G$-conjugate.

We now briefly describe the extension.  Set $s = r/2$.  Note that $T_{s+}/T_{r+} \cong \tor_{s+}/\tor_{r+}$ via the map $t \mapsto t-I$, and that all characters of the abelian group $\tor_{s+}/\tor_{r+}$ are of the form $X \mapsto \Psi(\Tr(YX))$ for some $Y\in \tor_{-r}$.  Now the restriction of $\phi$ to $T_{s+}$ defines a character of $T_{s+}/T_{r+}$, and so there exists some $\Aphi \in \tor_{-r}$, unique modulo $\tor_{-s}$, such that
\begin{equation}  \label{E:aphi}
\phi(t) = \Psi(\Tr(\Aphi (t-I))) \quad \textrm{for all $t\in T_{s+}$.}
\end{equation}
The genericity of $\phi$ is equivalent in this case to $\Aphi = \aphi X_{T}$ with $\val(\aphi \gamma_1) = -(r+y)$.
The formula \eqref{E:aphi} can also be used to define a character of $G_{y,s+}$, trivial on $G_{y,r+}$, which we'll denote $\Psi_\Aphi$.  The result is an extension of $\phi$ to a character $\hat{\phi} = \phi\Psi_\Aphi$ of $TG_{y,s+}$.  When $T$ is ramified, or when $T$ is unramified and $r$ is odd, we have $G_{y,s+} = G_{y,s}$, and so we may take $\rho = \hat{\phi}$.  Otherwise, that is, when $T$ is unramified and $r$ is even, then J.K.~Yu specifies a canonical construction, using the Weil representation, of an extension of $\hat{\phi}$ to a representation $\rho$ of $TG_{y,s}$ of degree $q$.

\subsection{$L$-packets of irreducible representations of $\SL_2(\ratk)$}
The $L$-packets of irreducible admissible representations of $G$ are given simply by $\GL_2(\ratk)$-conjugacy.  That is, for each (tame) irreducible admissible representation of $\GL_2(\ratk)$, its restriction to $\SL_2(\ratk)$ is the direct sum of one, two or four irreducible representations (see, for example, \cite{MoySally}), and these representations constitute an $L$-packet of $\SL_2(\ratk)$.  The $L$-packets are thus grouped into the following six classes:
\begin{itemize}
\item irreducible principal series representations: $\{\Ind_P^G \chi \}$ where $\chi$ is not a sign character;
\item reducible principal series representations: $\{H_\tau^+, H_\tau^-\}$, where $\Ind_{P}^G \sgn_\tau = H_\tau^+ \oplus  H_\tau^-$, for $\tau \in \{ \ep, \p, \ep \p\}$;
\item depth-zero special supercuspidal representations:  $\{ \cind_K^G \sigma_0^+$, $\cind_K^G\sigma_0^-$, $\cind_{K^\eta}^G (\sigma_0^+)^\eta$, $\cind_{K^\eta}^G (\sigma_0^-)^\eta\}$
\item other depth-zero supercuspidal representations: $\{ \cind_K^G \sigma, \cind_{K^\eta}^G \sigma^\eta\}$
\item positive depth unramified supercuspidal representations: $\{ \cind_{T_{1,\ep}G_{0,s}}^G \rho,$ $\cind_{T_{1,\ep}^\eta G_{1,s}}^G \rho^\eta\}$
\item ramified supercuspidal representations: $\{ \cind_{TG_{\frac12,s}}\rho,$ $\cind_{T^\beta G_{\frac12,s}}\rho^\beta\}$, where  $\beta = \smat{1 & 0\\0 & \ep} \in \GL_2(\ratk)$ and $T, T^\beta$ are ramified tori.
\end{itemize}
We return to these sets in Section~\ref{SS:46}.

\subsection{Some representations of $\SL_2(\R)$}
The irreducible representations of $K = \SL_2(\R)$ were determined by Shalika in his thesis \cite{Shalika}.  One can easily derive the representation theory of $K^\eta$ from this, as done in \cite{Nevins}.  In the present work, we have need only of the so-called \emph{ramified representations}, which we may describe as follows.  Our exposition differs from the original in that we parametrize the representation by depth rather than conductor (and so our indices are off by one), and that we replace with $\Psi$ the collection of choices of additive characters $\eta_k$.  

Let $\ell > 0$, $\ell \in \frac12\Z$.  For any choice of $(u,v)$ such that $\val(v)>\val(u)=0$, define $X = X_{u,v}$ to be the corresponding antidiagonal matrix \eqref{E:X}.  Then it is straightforward to see that the formula
$$
\Psi_X(z) = \Psi(\p^{-2\ell}\Tr(X(z-I))), \quad z\in G_{[0,\frac12],\ell}
$$
defines a character $\Psi_X$ of $G_{[0,\frac12],\ell}$ which is  trivial on $G_{0,2\ell+}$.  Note that such an $X_{u,v}$ is uniquely defined by $\Psi_X$ only modulo $\g_{[0,\frac12],\ell}$.  (It would be more natural to replace $X$ with $\p^{-2\ell}X$ as in Yu's construction, but this leads to more awkward notation.)  Let $C(X)$ denote the centralizer of $X$ in $K$; note that $C(X) = C(aX)$ for any $a\in \ratk^\times$, and that $C(X_{u,v}) = T_{u,v}$.   Given any character $\varphi$ of $C(X)$ agreeing with $\Psi_X$ on the intersection $C(X) \cap G_{[0,\frac12],\ell}$, $\varphi\Psi_X$ defines a character of $C(X)G_{[0,\frac12],\ell}$.  Shalika proved that the resulting induced representation
$$
\Sh_{2\ell}(\varphi,X) := \Ind_{C(X)G_{[0,\frac12],\ell}}^K \varphi \Psi_X
$$
is irreducible and has depth $2\ell$.  Each such representation has degree $\frac12 q^{2\ell-1}(q^2-1)$, and these exhaust all irreducible representations of $K$ whose degree is of this form for some $2\ell \in \mathbb{N}$.  Furthermore, if two such representations of the same depth are equivalent then their parameter pairs $(\varphi,X \mod \mathfrak{g}_{[0,\frac12],\ell})$ are conjugate under $K$.

\section{Branching Rules for Irreducible Representations of $\SL_2(\ratk)$} \label{S:BR}

The key tool in the restriction of representations of $\SL_2(\ratk)$ to $K = \SL_2(\R)$ is Mackey theory, whose analogue to the case of compactly induced representations was shown by Kutzko in \cite{KutzkoMackey}.  In the case of supercuspidal representations of $\SL_2(\ratk)$, this decomposition has the form
$$
\Res_K \cind_L^G \rho \cong \oplus_{\gamma \in K \backslash G / L} \Ind_{K \cap L^\gamma}^K \rho^\gamma
$$
and it is shown in \cite{Nevins2} that this decomposition is in many cases a decomposition into irreducibles.  

\subsection{Depth-zero representations}
Let us first consider the depth-zero supercuspidal representations $\cind_K^G\sigma$ and $\cind_{K^\eta}^G\sigma^\eta$.  Since $K \backslash G / K$ and $K \backslash G / K^\eta$ can each be represented by the set
$$
\left\{ \left. \alpha^{t} := \mat{\p^{-t} & 0 \\ 0 & \p^t} \right| t \geq 0 \right\},
$$
Mackey theory gives a decomposition indexed by $t \in \mathbb{N}$.  It is easy to see that the depth of $\Ind_{K \cap K^{\alpha^t}}^K\sigma^{\alpha^t}$ is $2t$ whereas that of $\Ind_{K\cap K^{\alpha^t\eta}}^G\sigma^{\alpha^t\eta}$ is $2t+1$.  

If we choose $\sigma \in \{\sigma_0^\pm\}$, then some character calculations \cite{Nevins2} reveal that these components of the Mackey decomposition intertwine with Shalika representations of the same depth and degree, hence are irreducible.
  More precisely, setting $\ch_0$ to be the central character of $\sigma_0^\pm$, we have
$$
\Res_K \cind_K^G \sigma_0^\pm \cong \sigma_0^\pm \oplus \left(\bigoplus_{t > 0} \Sh_{2t}(\ch_0, X_{u_{\pm}, 0})\right)
$$
where $u_\pm$ are elements of $\{1,\ep\}$ chosen so that $u_+ \equiv  -1$ and $u_- \equiv  -\ep$ modulo $(\R^\times)^2$.  Similarly, one has
\begin{equation} \label{E:eta0}
\Res_K \cind_{K^\eta}^G (\sigma_0^\pm)^\eta \cong \bigoplus_{t \geq 0} \Sh_{2t+1}(\ch_0, X_{u_{\pm}, 0}).
\end{equation}

The remaining depth-zero supercuspidal representations decompose into a great\-er number of irreducible constituents.  That is, each Mackey component is itself the direct sum of two inequivalent irreducibles of the same depth and same degree \cite{Nevins2}, and in fact these constituents are the same or similar to those appearing in the branching rules for the depth-zero special supercuspidal representations.  More precisely, denoting by $\ch$ the central character of the irreducible Deligne-Lustzig cuspidal $\sigma$, we have 
$$
\Res_K \cind_K^G \sigma \cong \sigma \oplus \bigoplus_{t > 0} \left(\Sh_{2t}(\ch, X_{1, 0}) \oplus \Sh_{2t}(\ch, X_{\ep, 0})\right)
$$
and
\begin{equation} \label{E:eta1}
\Res_K \cind_{K^\eta}^G \sigma^\eta \cong \bigoplus_{t \geq 0} \left( \Sh_{2t+1}(\ch, X_{1, 0}) \oplus  \Sh_{2t+1}(\ch, X_{\ep, 0})\right).
\end{equation}

These branching rules fit well with those of the depth-zero principal series representations from \cite{Nevins}.  In the case of principal series, the Mackey decomposition gives no information since $G=KP$.  Instead, one can filter $\Ind_P^G\chi$ by its $K_n$-invariant subspaces, where $K_n = G_{0,n}$ is the $n$th congruence subgroup.  Each of these subrepresentations has depth $n-1$, and is shown to contain precisely two irreducible subrepresentations of depth $d$ for each $1 \leq d < n$.

More precisely, for any depth-zero character $\chi$ of $\ratk^\times$, that is, a character which is trivial on $1+\PP$, one has
\begin{equation} \label{E:prinseries}
\Res_K \Ind_{P}^G \chi \cong (\Ind_{P}^G \chi)^{K_1} \oplus \bigoplus_{t\geq 1}\left( \Sh_t(\chi, X_{1,0}) \oplus  \Sh_t(\chi, X_{\ep,0}) \right).
\end{equation}
By well-known results in the representation theory of $\SL_2(\resk)$, the first component is irreducible (being the inflation of the corresponding principal series representation of $\SL_2(\resk)$) unless $\chi$ restricts to either the trivial character $\triv$ or the sign character $\sgn$ on $\R^\times$.   In the first case, $(\Ind_{P}^G \chi)^{K_1}$ decomposes as $\triv \oplus \St$, where $\St$ denotes the $q$-dimensional Steinberg representation of $\SL_2(\resk)$; in the second case, it decomposes as a direct sum of two inequivalent representations $\Xi_\sgn^\pm$ of the same degree $\frac12 (q+1)$.

With respect to our normalized induction, the three (equivalence classes of) reducible principal series are given by choosing $\chi = \sgn_\tau$, where $\tau$ represents any of the three nontrivial square classes of $\ratk^\times/(\ratk^\times)^2$.  Note that $\Res_{\R^\times}\sgn_\ep = \triv$ and $\Res_{\R^\times} \sgn_{a\p} = \sgn$ for $a\in \R^\times$.  Denote the decomposition of $\Ind_P^G \sgn_\tau$ into irreducibles by  $\pi_\tau^+ \oplus \pi_\tau^-$, where $\pi_\tau^+$ lives on a subspace $H_\tau^+$ consisting of functions supported on the set $\{u\in \ratk^\times \mid \sgn_\tau(u)=1\}$ in a natural realization \cite[Ch2\S3.5]{GGPS}.  Setting $\diamond =+$ if $-1 \notin (\ratk^\times)^2$ and $\diamond=-$ otherwise, we have \cite{Nevins}
\begin{equation} \label{E:redprinseries}
\Res_K (\pi_\tau^+) \cong 
\begin{cases}
\triv \oplus \bigoplus_{t>0} \left(\Sh_{2t}(1, X_{1,0}) \oplus \Sh_{2t}(1, X_{\ep,0})\right) &\textrm{if $\tau = \ep$}\\
\Xi_\sgn^\diamond \oplus \bigoplus_{n\geq 1} \Sh_{n}(\sgn_\tau,X_{1,0})
 & 
 \textrm{if $\tau = -\p$,}\\
\Xi_\sgn^{\diamond} \oplus \bigoplus_{t > 0} \left( \Sh_{2t-1}(\sgn_\tau,X_{\ep,0}) \oplus \Sh_{2t}(\sgn_\tau,X_{1,0})\right) & 
\textrm{if $\tau = -\ep \p$},
\end{cases}
\end{equation}
with $\Res_K(\pi_\tau^-)$ given by the complement of \eqref{E:redprinseries} in \eqref{E:prinseries}.

We make the following observation.

\begin{observation} \label{O:depthzero}
The irreducible components of positive depth arising in the branching rules of a depth-zero representation $\pi$ of $G$ are chosen from the set
$$
\{ \Sh_{t}(\ch, X_{1,0}), \Sh_t(\ch,X_{\ep,0}) \mid t\geq 1\}
$$
where $\ch$ denotes the central character of $\pi$.
\end{observation}

\subsection{Positive depth representations} \label{SS:pos}
The branching rules for irreducible representations of $\SL_2(\ratk)$ of positive depth follow several different patterns.

Let us first consider the case of a supercuspidal representation $\pi$ of $G$ of positive depth $r$ associated to the datum $(T,y,\phi)$.  Set $s=r/2$.   Choose $\Aphi= \aphi \p^{-\lrc{r+y}} X_{1,\gamma_1^{-1}\gamma_2}$ representing $\phi$ on $T_{s+}$. 
To determine the Mackey decomposition of the restricted representation $\Res_K\pi = \Res_K\cind_{TG_{y,s}}^G \rho$, we first note that if $\Lambda$ is a set of coset representatives for $(K \cap K^{\alpha^{-t}})\backslash K /TG_{y,s}$, then $\{ I, \alpha^t\lambda \mid \lambda \in \Lambda, t>0\}$ is a set of representatives for the Mackey double coset space $K \backslash G / TG_{y,s}$.

If $T = T_{1,\ep}$ and $y=0$, then $(K \cap K^{\alpha^{-t}})\backslash K /TG_{y,s}$ is represented by $\Lambda = \{I, E\}$ where $E = \smat{u & v \\ v &\ep^{-1}u }$ for some $(u,v) \in \R^\times \times \R^\times$ satisfying $u^2-v^2\ep = \ep$.  For $\lambda \in \Lambda$, the representations $\rho^{\alpha^t\lambda}$ each have depth $r+2t$.  One can show that the corresponding component 
$\Ind_{K \cap (TG_{y,s})^{\alpha^t\lambda}}^K  \rho^{\alpha^t\lambda}$ can be constructed from Shalika data $(\phi^{\alpha^t\lambda},\p^{-(r+2t)}\Aphi^{\alpha^t\lambda})$; this argument requires making a choice of representative $\Aphi$ modulo $\tor_{(-s)+}$ in the case that $G_{y,s} \neq G_{y,s+}$ \cite{Nevins2}.  It follows that Mackey theory gives a decomposition of $\displaystyle \Res_K \cind_{TG_{0,s}}^G \rho$ into irreducibles as
\begin{equation} \label{E:unram}
\Ind_{TG_{0,s}}^K \rho 
\oplus \bigoplus_{t>0}
\left( \Sh_{r+2t}(\phi^{\alpha^{t}},\aphi X_{1,\ep\p^{4t}}) \oplus 
\Sh_{r+2t}(\phi^{\alpha^{t}E},\aphi X_{\ep,\p^{4t}}) \right).
\end{equation}
On the other hand, if $T = T_{\p^{-1},\ep\p}$ and $y=1$, then a corresponding set of double coset representatives is $\Lambda = \{I,E^\eta\}$.  As in the depth-zero case, the depths of the components occuring at $y=1$ are offset by one from those occuring at $y=0$.  We have therefore that $\displaystyle \Res_K \cind_{TG_{1,s}}^G \rho$ decomposes into irreducibles as
\begin{equation} \label{E:unram2}
\bigoplus_{t\geq 0}\left( \Sh_{r+2t+1}(\phi^{\alpha^{t}},\aphi X_{1,\ep\p^{4t+2}}) \oplus \Sh_{r+2t+1}(\phi^{\alpha^{t}E^\eta},\aphi X_{\ep,\p^{4t+2}}) \right).
\end{equation}
Thus the decomposition of unramified supercuspidal representations continues the pattern established for depth-zero supercuspidal representations, with pairs of components occuring only for every other depth after $r$.

The decomposition of the ramified supercuspidal representations follows the pattern established for some reducible principal series instead.  Namely, let $T = T_{\gamma_1,\gamma_2}$ be a ramified torus as in Table~\ref{Table:tori} and set $y=\frac12$.  Then we may take $\Lambda = \{I, w\}$ where $w = \smat{0& 1\\-1&0}$.  This time $\Ind_{K \cap (TG_{y,s})^{\alpha^t}}^K \rho^{\alpha^t}$ has depth $r+2t+\frac12$ whereas $\Ind_{K \cap (TG_{y,s})^{\alpha^tw}}^G \rho^{\alpha^tw}$ has depth $r+2t-\frac12$.  We thus obtain a single component of each integral depth greater than $r$.  Following the same argument as in the unramified case, we find that these components are each irreducible and constructible from Shalika data which is a conjugate of $(\phi,\Aphi)$.  We have \cite{Nevins2}
\begin{align} \label{E:ram}
\Res_K & \cind_{TG_{\frac12,s}}^G \rho \cong  \Sh_{r+\frac12}(\phi, \aphi X_{\gamma_1,\gamma_2}) \oplus \\
\notag &\bigoplus_{t\geq 0} \left(
\Sh_{r+2t-\frac12}(\phi^{\alpha^tw}, \aphi X_{-\gamma_2\p^{-1},-\gamma_1\p^{4t-1}})
\oplus 
\Sh_{r+2t+\frac12}(\phi^{\alpha^t}, \aphi X_{\gamma_1,\gamma_2\p^{4t}})
\right).
\end{align}

Finally, we consider the principal series of positive depth.  Let $\chi$ be a character of $\ratk^\times$ of integral depth $r>0$.  Then for any $\gamma \in \R$,  $\chi$ defines a character of $T_{1,\gamma^2}$ via $\chi_\gamma(t(a,b)) = \chi(a+b\gamma)$.  In \cite[Lemma 7.1]{Nevins} we define a scalar $\lambda = \lambda_\chi \in \R^\times$, uniquely defined modulo $\PP^{\lrc{s+}}$, derived from the restriction of $\chi$ to $1 + \PP^{\lrc{s}}$.  If one sets $u_0 = \lambda\p^t$ and $u_1 =  \ep^{-1}u_0$, then we have that $\displaystyle \Res_K \Ind_P^G \chi$ decomposes into irreducibles as
\begin{equation} \label{E:principalseries}
(\Ind_{P}^G \chi)^{G_{0,r+}} \oplus \bigoplus_{t>0} \left( 
\Sh_{r+t}(\chi_{u_0},X_{1, u_0^2}) \oplus \Sh_{r+t}(\chi_{u_1},X_{\ep,\ep u_1^2}) \right).
\end{equation}

We conclude with the following observation.

\begin{observation} \label{O:posdepth}
If $\pi$ is an irreducible representation of $G$ of positive depth $r$, then the pattern of the number of irreducible components of $\Res_K(\pi)$ of depth greater than $r$ identifies the class of $\pi$ among unramified supercuspidal, ramified supercuspidal and principal series representations.  Namely, the unramified supercuspidal representations give two irreducible components at every other depth; the ramified supercuspidal representations give one irreducible component at each depth; and the principal series representations give two irreducible components at each depth.
\end{observation}

\section{Main Theorems and Conclusions} \label{S:conclusion}


\subsection{Tail ends} Casselman and Silberger observed that for the groups $\GL_2(\ratk)$ and $\PGL_2(\ratk)$, respectively, the restriction of any two irreducible representations to a maximal compact subgroup can differ by at most a finite-dimensional piece.  While this is not the case for $\SL_2(\ratk)$, their results do imply that all but finitely many of the irreducible $K$-representations occuring in the restriction of any irreducible representation of $\SL_2(\ratk)$ must come from a common library.  Given the above explicit branching rules, we can elaborate on this property.

\begin{theorem} \label{T:tailends}
Let $\pi$ be an irreducible representation of $G$ of depth $r$.  Then the irreducible components of $\Res_K\pi$ of depth greater than $2r$ 
are chosen from the set
$$
\{ \Sh_{d}(\ch, X_{1,0}), \Sh_d(\ch,X_{\ep,0}) \mid d>2r\}
$$
where $\ch$ denotes the central character of $\pi$.
\end{theorem}

\begin{proof}
This has already been established in the depth-zero case (Observation~\ref{O:depthzero}).

Recall that the representation $\Sh_{2\ell}(\varphi, X)$ depends only on the $K$-conjugacy class of $X$ modulo $\mathfrak{g}_{[0,\frac12],\ell}$, and for equal choices of $X$, two such representations are equivalent only if the characters $\varphi$ are equal on $C(X)$.  Note that if $C(X) = T_{1,\gamma}$, where $\val(\gamma)=v>0$, then the elements $t(a,b)\in T_{1,\gamma}$ satisfy $a \equiv \pm 1\mod \PP^v$. 

First let $(\pi,V)$ be a principal series representation with branching rules as in  \eqref{E:principalseries}.  In each component of depth $r+t$, $t>0$, we have $\varphi = \chi_\gamma$ for some $\gamma$ of valuation $t$.  Now $\chi_\gamma(t(a,b)) = \chi(a+b\gamma) = \chi(aI)$ for all $t(a,b)\in T_{1,\gamma^2}$ if and only if $t=\val(\gamma)>r$, in which case $\chi_\gamma$ acts by the central character $\ch$ of $\chi$.   Furthermore, when $t>r$, we see that $X_{1, (\lambda\p^{t})^2}$ and $X_{\ep,\ep (\ep^{-1}\lambda\p^{t})^2}$ are equivalent modulo $\mathfrak{g}_{[0,\frac12],s+t/2}$ to  $X_{1,0}$ and $X_{\ep,0}$, respectively.  It follows that
$$
V / V^{G_{0,(2r)+}} \cong \bigoplus_{d > 2r} \left(\Sh_{d}(\ch,X_{1,0}) \oplus \Sh_d(\ch,X_{\ep,0}) \right).
$$

Next let $(\pi,V)$ be an unramified supercuspidal representation, with branching rules as in \eqref{E:unram} or \eqref{E:unram2}.  First let $y=0$ and $T = T_{1,\ep}$; we verify directly that $C(\aphi X_{1,\ep\p^{4t}})^{\alpha^t}$ and  $C(\aphi X_{\ep,\p^{4t}})^{\alpha^tE}$ are simply equal to $ZT_{2t}$.  Thus $\phi^{\alpha^t}$ and $\phi^{\alpha^tE}$ are given by the central character $\ch$ of $\phi$ if and only if $2t > r$, or when $r+2t > 2r$. 
 Similarly, when $y=1$ and $T = T_{\p^{-1},\ep\p}$, we have $C(\aphi X_{1,\ep\p^{4t+2}})^{\alpha^t}$ and $C(\aphi X_{\ep,\p^{4t+2}})^{\alpha^tE^\eta}$ are simply equal to $ZT_{2t+1}$, so we may replace the characters $\phi^{\alpha^t}$ and $\phi^{\alpha^tE^\eta}$ by the central character $\ch$ if and only if $2t+1>r$, or $r+2t+1>2r$.  As above, we note that the elements $\aphi X_{u,v} = X_{\aphi u, \aphi v}$ appearing as parameters in \eqref{E:unram} and \eqref{E:unram2} are equivalent modulo $\mathfrak{g}_{[0,\frac12],s+t/2}$ to $X_{\aphi u, 0}$.  Thus up to $K$-conjugacy, which allows us to scale $X_{u,0}$ by a square in $\R^\times$, we may replace each $\aphi u$ by a simpler equivalent representative in $\{1,\ep\}$.  We deduce that we have
$$
V / V^{G_{0,(2r)+}} \cong \bigoplus_{t > s+y/2} \left(\Sh_{r+2t+y}(\ch,X_{1,0}) \oplus \Sh_{r+2t+y}(\ch,X_{\ep,0}) \right).
$$

Finally, we let $(\pi,V)$ be a ramified supercuspidal representation, with branching rules as in \eqref{E:ram}.  Since $C(X_{\gamma_1,\gamma_2\p^{4t}})^{\alpha^t} = Z(T_{\gamma_1,\gamma_2})_{2t+\frac12}$, $\phi^{\alpha^t}$ is given by the central character $\ch$ of $\phi$ for $2t+\frac12>r$.  Similarly, $\phi^{\alpha^tw}$ reduces to $\ch$ for $2t-\frac12 > r$.  As before, for this range of $t$, we may replace $\aphi X_{\gamma_1,\gamma_2\p^{4t}}$ by $X_{\aphi \gamma_1,0}$ and $\aphi X_{-\gamma_2\p^{-1},-\gamma_1\p^{4t-1}}$ by $X_{-\aphi \gamma_2 \p^{-1},0}$.   Considering the possibilities for $(\gamma_1,\gamma_2)$ given in Table~\ref{Table:tori}, we see that if $-1\in (\ratk^\times)^2$ and $T$ splits over $\ratk(\sqrt{\p})$, or if $-1\notin (\ratk^\times)^2$ and $T$ splits over $\ratk(\sqrt{\ep\p})$ (in brief: if $T$ splits over $\ratk(\sqrt{-\p})$)  then $\aphi \gamma_1$ and $-\aphi \gamma_2\p^{-1}$ lie in the same square class of $\R^\times$.  This yields
$$
V / V^{G_{0,(2r)+}} \cong \bigoplus_{d>2r} \Sh_{d}(\ch, X_{z,0})
$$
where $z \in \{1,\ep\}$ represents the square class of $\aphi \gamma_1$.

On the other hand, if $T$ splits over $\ratk(\sqrt{-\ep\p})$, then we instead have
\begin{equation} \label{E:z}
V / V^{G_{0,(2r)+}} \cong \bigoplus_{d>2r} \Sh_d(\ch, X_{z(d),0})
\end{equation}
where $z$ is a function of the parity of $d$ taking values in $\{1,\ep\}$, such that $z(r+\frac12)$ is in the square class of $\aphi \gamma_1$.

This completes the proof.
\end{proof}

In particular, we see that the tail ends of the branching rules for supercuspidal representations of $\SL_2(\ratk)$ contain more information about the inducing representation than do those for $\GL_2(\ratk)$; one detects not only the central character of $\pi$ but also in many cases its class.  From the proof of Theorem~\ref{T:tailends} we recover the following precise version of this statement.

\begin{scholium} \label{S:1}
If $(\pi,V)$ and $(\pi',W)$ are two irreducible representations of $G$ of positive (but not necessarily equal) depth, such that for some $t > 0$, $V/V^{G_{0,t+}} \cong W/W^{G_{0,t+}}$ as representations of $K=G_0$, then $\pi$ and $\pi'$
\begin{itemize}
\item have the same central character and 
\item are either both principal series representations, or both supercuspidal representations corresponding to tori which split over the same field.
\end{itemize}
\end{scholium}

We remark that this does not extend to the depth-zero case: the three reducible principal series, whose branching rules were given in \eqref{E:redprinseries}, have decompositions whose tail ends mimic those of supercuspidal representations corresponding to the three different splitting fields of tori.  They are distinguished as occuring as constituents of principal series only by their \emph{leading term} (that is, the component(s) of the branching rules of least depth).

\subsection{$K$-intertwining}  \label{SS:42}

Having established when the tail ends of the branching rules of two representations of $G$ are equal, one may ask more generally when two representations will intertwine as representations of $K$.   For example, we see that this can occur if they have the same central character, they arise from tori with different splitting fields and, in the case of two unramified representations, have depths of opposite parity.  

We have the following converse to Scholium~\ref{S:1}.

\begin{corollary}
Let $(\pi,V),(\pi',W)$ be two representations of positive depth associated to the same torus $T$ (split, in the case of principal series, or anisotropic in the case of supercuspidal representations), with the same central character.  In the case of supercuspidal representations, let $\aphi, \aphi' \in \R^\times$ denote the scalars arising in the restriction of $\phi$ (respectively, $\phi'$) to $T_{s+}$ (as in Section~\ref{SS:pos}).  Then $\Res_K\pi$ and $\Res_K\pi'$ intertwine if any of the following conditions hold:
\begin{enumerate}[(a)]
\item $T$ is split;
\item $T$ is unramified and the depths of $\pi$ and $\pi'$ are of equal parity;
\item $T$ is ramified, splitting over $\ratk(\sqrt{-\p})$, and $\aphi \equiv \aphi' \mod (\R^\times)^2$;
\item $T$ is ramified, splitting over $\ratk(\sqrt{-\ep\p})$, and the associated functions $z$ defined in \eqref{E:z} are equal.
\end{enumerate}
If in addition we have $\aphi=\aphi'$ then we have $V/V^{G_{0,(r+d)+}} \cong W/W^{G_{0,(r+d)+}}$, where  $d$ the depth of $\phi^{-1}\phi'$ (or of $\chi^{-1}\chi'$ in the split case).
\end{corollary}

\begin{proof}
In each of the cases enumerated, one sees that the tail ends of the branching rules will coincide.  For the final assertion, note that the additional hypothesis implies that $X=X'$.  Given that $\Sh_d(\varphi,X) \cong \Sh_d(\varphi',X)$ only if $\varphi = \varphi'$, and these characters are given here by $\phi^{\lambda}$ and ${\phi'}^{\lambda}$ respectively, for $\lambda$ running over the Mackey coset representatives, we may repeat the argument used in the proof of Theorem~\ref{T:tailends} to deduce the equivalence of all corresponding components of depth greater than $r+d$.
\end{proof}


We note from this discussion that the degree of intertwining between the restrictions of two representations is another tool for identifying the class of one representation relative to another.

\subsection{Distribution of $K$-representations}
The leading terms of principal series representations were shown in \cite{Nevins} to correspond to split elements $X$  via the Shalika classification.  The construction of $\rho$ in the case of unramified supercuspidal representations mimics that used by Shalika to construct the representations corresponding to unramified elements $X$ in \cite{Shalika}, so the leading terms in the unramified case correspond to so-called unramified representations in the Shalika classification.  We've shown here that the leading terms of ramified supercuspidal representations are ramified Shalika representations.  

Unsurprisingly, all representations of $K$ appear in the restriction of some representation of $G$, although their distribution is not uniform.  Except for the leading terms, every $K$-irreducible occuring in the branching rules for any irreducible representation $\pi$ of $G$ is a so-called ramified representation of $K$.

That said, note that the centralizer $C(X_{\gamma_1,\gamma_2})$ appearing in the construction of Shalika's ramified representations can be, when $\gamma_1\gamma_2\neq 0$, any type of split, ramified or unramified torus, depending on the square class of $\gamma_1\gamma_2$.  In fact, our derivation of the Shalika data from the supercuspidal $G$-data implies that $C(X) = C(\Aphi^\gamma) = K \cap C(\Aphi)^\gamma = K\cap T^\gamma$ for some $\gamma \in \{I, \alpha^t, \alpha^t\lambda \mid \lambda \in \Lambda, t>0\}$, which will have the same splitting field as $T$.  Alternately, one can observe this directly from the branching rules.

\subsection{Leading terms}
We note that the restriction to $K$ of a depth-zero supercuspidal representation $\pi$ induced from $K$ yields as its leading term  the inducing representation.  Consequently, $\pi$ may be entirely recovered from the leading term of its branching rules.  Similarly, transitivity of induction implies that the leading terms of positive depth supercuspidal representations constructed from tori contained in $K$ induce to give the full representation.

The situation for principal series is quite different; one can hope at best to recover $\Res_K \chi$, rather than all of $\chi$.  This is in the spirit of the theory of types, however.  It was observed in \cite{Nevins} that the inducing datum for the leading term of a principal series representation of $G$ is a Bushnell-Kutzko type for the representation; this was also observed for principal series of $\GL_3(\ratk)$ in \cite{NC1}.  

For those supercuspidal representations induced from tori contained in $K^\eta$ but not $K$, however, there is a finite amount of information lost in restricting to $K$.  Namely, in the depth-zero case, one recovers only the central character of the inducing cuspidal representation; and in positive depth cases, there are no components in which the character $\phi$ appears in its entirety as it only appears as its restriction to $K$.   It follows that to mitigate such information loss, one should consider the restriction to $K^\eta$ as well.

\subsection{Restriction to $K^\eta$}
In \cite{Nevins}, the representation theory of $K^\eta$ was defined relative to that of $K$ so as to permit an equally explicit formulation of the branching rules for $\Res_{K^\eta}\pi$, for $\pi$ a principal series representation.   One discovered that the restriction of principal series to $K^\eta$ gave an entirely analogous decomposition into irreducibles, one which reproduced in all but the cases of reducible principal series the information gleaned from the restriction to $K$.

In the case of supercuspidal representations, on the other hand, the two restrictions will not be entirely redundant.   It is easy to see, using the Mackey decomposition as above, that the restriction to $K^\eta$ of a depth-zero supercuspidal representation induced from $K^\eta$ will yield as its leading term the inducing representation.  

For a more general group $G$, one anticipates a similar rationing of data about supercuspidal representations among the various conjugacy classes of maximal compact subgroups.  Furthermore, the use of all conjugacy classes of maximal compact subgroups will ensure that one always has some leading component containing a type for the representation.

\subsection{$L$-packets} \label{SS:46}
Let $\Pi$ be an irreducible representation of $\GL_2(\ratk)$.  The restriction of $\Pi$ to $\GL_2(\R)$ is known (by Hansen \cite{Hansen} for supercuspidal representations, and by Silberger \cite{Silberger} for principal series representations (of $\PGL_2(\ratk)$)) to admit a single irreducible representation at each depth greater than or equal to the depth of $\Pi$.   

Given that the $L$-packet corresponding to $\Pi$ consists of the irreducible summands of $\Res_{\SL_2(\ratk)}\Pi$, it follows that the sum of the branching rules of these irreducible summands should give the restriction to $\SL_2(\R)$ of $\Pi$.  Thus we observe (a fact which can also be shown directly) that each such $\GL_2(\R)$-representation decomposes as exactly two irreducible representations of $\SL_2(\R)$ upon restriction.  We deduce that the patterns in the branching rules are in part a consequence of the structure of the $L$-packets.

\subsection{Branching rules of supercuspidal representations and the representation theory of $K$}  \label{SS:last}

The construction of supercuspidal representations via induction from compact open subgroups necessarily provides a construction of certain irreducible representations of maximal compact subgroups, via transitivity.  That is, when $TG_{y,s} \subseteq K$, $\Ind_{TG_{y,s}}^K\rho$ is irreducible.

In the case of $\SL_2(\ratk)$, we see that these form only a small portion of the representations of $K$, namely those occuring as a leading term of the branching rules of a supercuspidal representation.  Another set --- the leading terms of principal series representations --- are obtained via parabolic induction in the groups $\SL_2(\R/\PP^n)$, which gives a simple realization of the $K_n$-invariant vectors of the principal series representation of depth $n-1$.

The remaining representations of $K$ are obtained in some sense by conjugation of the inducing datum of the leading terms, an effect which stresses the key nature of these leading terms not only in the representation theory of $G$, as in the discussion above, but also in the representation theory of $K$.


\begin{thebibliography}{99999}

\bibitem[BK93]{BushnellKutzko} Colin J.~Bushnell, Philip C.~Kutzko,
\emph{The admissible dual of ${\rm GL}(N)$ via compact open subgroups.}
Annals of Mathematics Studies, 129. Princeton University Press, Princeton, NJ, 1993.

\bibitem[Ca77]{Cartier} P.~Cartier, ``Representations of $p$-adic groups: a survey,''  \emph{Automorphic forms, representations and $L$-functions (Corvallis, Ore., 1977), Part 1,} pp. 111--155, Proc. Sympos. Pure Math., XXXIII, Amer. Math. Soc., Providence, R.I., 1979. 


\bibitem[CN09]{NC1} Peter S.~Campbell, Monica Nevins, ``Branching rules for unramified principal series representations of ${\rm GL}(3)$ over a $p$-adic field,''  J. Algebra  321  (2009),  no. 9, 2422--2444.

\bibitem[CN10]{NC2} Peter S.~Campbell, Monica Nevins, ``Branching rules for ramified principal series representations of $\rm{GL}(3)$ over a $p$-adic field,''  Canad. J. Math.  62  (2010),  no. 1, 34--51.


\bibitem[Cs73]{Casselman} William Casselman,
``The restriction of a representation of ${\rm GL}_{2}(k)$ to ${\rm GL}_{2}({\mathfrak{o}})$,''
Math. Ann. 206 (1973), 311--318. 

\bibitem[DM91]{DigneMichel}
Fran\c{c}ois Digne and Jean Michel, \emph{Representations of
Finite Groups of Lie Type}, London Mathematical Society 
Student Texts {\bf 21}, Cambridge University Press 1991.


\bibitem[GGPS]{GGPS} I.M.~Gel'fand, M.I.~Graev, I.I.~Pyatetskii-Shapiro,
\emph{Representation theory and automorphic functions.}
Translated from the Russian by K. A. Hirsch. Generalized Functions, 6. Academic Press, Inc., Boston, MA, 1990.

\bibitem[HM08]{HakimMurnaghan} Jeff Hakim, Fiona Murnaghan, ``Distinguished Tame Supercuspidal Representations,''
Int. Math. Res. Pap. IMRP 2008, no. 2, Art. ID rpn005, 166 pp. 

\bibitem[Ha87]{Hansen} Kristina Hansen,  ``Restriction to ${\rm GL}\sb 2({\mathcal{O}})$ of supercuspidal representations of ${\rm GL}\sb 2(F)$,''  Pacific J. Math.  130  (1987),  no. 2, 327--349.

\bibitem[Ho77]{Howe} Roger E.~Howe, 
``Kirillov theory for compact $p$-adic groups,''
Pacific J. Math. 73 (1977), no. 2, 365--381. 


\bibitem[Ki07]{Kim} Ju-Lee Kim, ``Supercuspidal representations: an exhaustion theorem,''  J. Amer. Math. Soc.  20  (2007),  no. 2, 273--320.

\bibitem[Ki09]{KimMonograph} Ju-Lee Kim, ``Supercuspidal representations: construction and exhaustion.''  \emph{Ottawa lectures on admissible representations of reductive $p$-adic groups},  79--99, Fields Inst. Monogr., 26, Amer. Math. Soc., Providence, RI, 2009.

\bibitem[Ku77]{KutzkoMackey} P.~C.~Kutzko, ``Mackey's theorem for nonunitary representations,''  Proc. Amer. Math. Soc.  64  (1977), no. 1, 173--175. 




\bibitem[KS83]{KutzkoSally} P.C.~Kutzko, P.J.~Sally, Jr.,``All supercuspidal representations of ${\rm SL}\sb{l}$ over a $p$-adic field are induced,'' \emph{Representation theory of reductive groups (Park City, Utah, 1982)}, 185--196, Progr. Math., 40, BirkhÃ¤user Boston, Boston, MA, 1983. 


\bibitem[Mo99]{Morrislevel0} Lawrence Morris, ``Level zero $\textbf{G}$-types,''  Compositio Math.  118  (1999),  no. 2, 135--157.


\bibitem[MP94]{MoyPrasadKtypes} Allen Moy, Gopal Prasad,
``Unrefined minimal $K$-types for $p$-adic groups,''
Invent. Math. 116 (1994), no. 1-3, 393--408. 

\bibitem[MP96]{MoyPrasadJacquet}  Allen Moy, Gopal Prasad,
``Jacquet functors and unrefined minimal $K$-types,''  Comment. Math. Helv.  71  (1996),  no. 1, 98--121. 

\bibitem[MS84]{MoySally}  Allen Moy, Paul J.~Sally, Jr. ``Supercuspidal representations of ${\rm SL}\sb{n}$ over a $p$-adic field: the tame case,'' Duke Math. J. 51 (1984), no. 1, 149--161.  

\bibitem[Ne05]{Nevins}  Monica Nevins, ``Branching rules for principal series representations of ${\rm SL}(2)$ over a $p$-adic field,''  Canad. J. Math.  57  (2005),  no. 3, 648--672.

\bibitem[Ne10]{Nevins2}  Monica Nevins, ``Branching rules for supercuspidal representations of ${\rm SL}(2)$ over a $p$-adic field,'' preprint.

\bibitem[Sa98]{SallyPhysics}  Paul J.~Sally, Jr. ``An introduction to $p$-adic fields, harmonic analysis and the representation theory of ${\rm SL}_2$,''
Lett. Math. Phys. 46 (1998), no. 1, 1--47.

\bibitem[Si70]{Silberger} Allan J.~Silberger,  \emph{${\rm PGL}_{2}$ over the $p$-adics: its representations, spherical functions, and Fourier analysis,}  Lecture Notes in Mathematics, Vol. 166 Springer-Verlag, Berlin-New York 1970.

\bibitem[Si77]{Silberger2}  Allan J.~Silberger, ``Irreducible representations of a maximal compact subroup of $\PGL_2$ over the $p$-adics,'' Math. Ann. 229 (1977), 1--12.

\bibitem[Sh67]{Shalika} Joseph A.~Shalika, 
``Representation of the two by two unimodular group over local fields,'' PhD thesis work (1967) reprinted in \emph{Contributions to automorphic forms, geometry, and number theory: A Volume in Honor of Joseph A. Shalika}, 1--38, Johns Hopkins Univ. Press, Baltimore, MD, 2004. 

\bibitem[St08]{Stevens} Shaun Stevens, ``The supercuspidal representations of $p$-adic classical groups,'' Invent. Math. 172 (2008), no. 2, 289--352. 

\bibitem[Ti77]{Tits} J.~Tits,  ``Reductive groups over local fields,'' \emph{Automorphic forms, representations and $L$-functions (Corvallis, Ore., 1977), Part 1,} pp. 29--69, Proc. Sympos. Pure Math., XXXIII, Amer. Math. Soc., Providence, R.I., 1979.


\bibitem[Tu80]{Tuvell}  Walter E.~Tuvell, ``Harmonic Analysis of the Principal Series of $\SL_2$ over a $p$-adic field,'' PhD Thesis, University of Chicago, 1980.


\bibitem[Vo77]{Vogan} David A.~Vogan, Jr.,
``Classification of the irreducible representations of semisimple Lie groups,''
Proc. Nat. Acad. Sci. U.S.A. 74 (1977), no. 7, 2649--2650. 

\bibitem[Vo07]{Vogan2} David A.~Vogan, Jr.,
``Branching to a maximal compact subgroup,'' \emph{Harmonic analysis, group representations, automorphic forms and invariant theory}, 321--401,
Lect. Notes Ser. Inst. Math. Sci. Natl. Univ. Singap., 12, World Sci. Publ., Hackensack, NJ, 2007. 

\bibitem[Yu01]{Yu} Jiu-Kang Yu,
``Construction of tame supercuspidal representations,''
J. Amer. Math. Soc. 14 (2001), no. 3, 579--622.


\end{thebibliography}
\end{document}